\documentclass[11pt]{article}
\usepackage[a4paper,margin=1in]{geometry}
\usepackage{amsmath,amssymb,amsthm,mathtools}
\usepackage{microtype}
\usepackage{hyperref}
\usepackage{booktabs}
\usepackage{pgfplots}
\pgfplotsset{compat=1.18}

\title{Worst-case Nonparametric Bounds for the Student T-statistic}
\author{David Edelman \\ University College Dublin}
\date{\today}

\theoremstyle{plain}
\newtheorem{theorem}{Theorem}
\newtheorem{lemma}{Lemma}

\theoremstyle{remark}

\newtheorem{example}{Example}

\newcommand{\RR}{\mathbb{R}}
\newcommand{\PP}{\mathbb{P}}
\newcommand{\PPstar}{\PP^{\ast}}

\def\med{\mathrm{med}}
\begin{document}
\maketitle

\begin{abstract}
We address the problem of finding worst-case nonparametric bounds for T-statistic by considering the extremal problem of maximising the mid-quantile (a special case of 'smoothed quantile' as discussed in \cite{St77} and \cite{W11}) $\tilde Q(S(w);\alpha)$ over nonnegative weight vectors $w\in\RR^n$ with $\|w\|_2=1$, where $S(w)=\sum_{i=1}^n w_i \varepsilon_i$ and $\varepsilon_i$ are independent Rademacher variables. While classical results of Hoeffding [1] and Chernoff [2] may be used to provide sub-Gaussian upper bounds, and optimal-order inequalities were later obtained by the author [3,4], the associated extremal problem has remained unsolved. We resolve this problem exactly (for the Mid-Quantile and, trivially, the Continuous case): for each $\alpha<{1\over 2}$ and each $n$, we determine the maximal value and characterise all maximising weights. The maximisers are $k$-sparse equal-weight vectors with weights $1/\sqrt{k}$, and the optimal support size $k$ is found by a finite search over at most $n$ candidates. This yields an explicit envelope $M_n(\alpha)$ and its universal limit as $n$ grows. Our results provide exact solutions to problems that have been studied through bounds and approximations for over sixty years, with applications to nonparametric inference, self-standardised statistics, and robust hypothesis testing under symmetry assumptions, including a conjecture by Edelman\cite{edelman1990}, albeit for continuous distributions only (which he did not specify, which has been found to not always hold otherwise)
\end{abstract}

\section{Introduction}

Let $\varepsilon_1, \varepsilon_2, \ldots, \varepsilon_n$ be independent Rademacher random variables, each taking values $\pm 1$ with equal probability. For a weight vector $w \in \RR^n$ with $\|w\|_2 = 1$, define the weighted Rademacher sum
\begin{equation}\label{eq:rademacher_sum}
S(w) := \sum_{i=1}^n w_i \varepsilon_i
\end{equation}
and Mid-Quantile function
\begin{equation}\label{eq:rademacher_sum2}
\tilde Q(S(w);\alpha)=\max \{t: P(S(w)>t)+{1\over 2}P(S(w)=t)\}\ge \alpha
\end{equation}

The central problem addressed in this paper is the exact determination of
\begin{equation}\label{eq:main_problem}
M_n(t) := \sup_{\substack{w\in\RR_+^n\\ \|w\|_2=1}} \tilde Q(S(w);\alpha),
\end{equation}
together with the characterisation of all weight vectors achieving this supremum.

This extremal problem sits at the intersection of probability theory, concentration inequalities, and statistical inference. While upper bounds for tail probabilities of weighted sums have been extensively studied since the pioneering work of Hoeffding [1] and Chernoff [2], the exact maximisation problem has remained open. Our main contribution is to solve this problem completely, providing explicit formulas for the maximum values and characterising all optimal weight configurations.

\subsection{Historical Context and Motivation}

The study of tail probabilities for sums of independent random variables has a rich history spanning over a century. The foundational work of Hoeffding [1] in 1963 established the now-classical sub-Gaussian bound
\begin{equation}\label{eq:hoeffding}
\PP(S(w) \ge t) \le \exp\left(-\frac{t^2}{2}\right)
\end{equation}
when $\|w\|_2 = 1$, which has become a cornerstone of modern probability theory with over 12,000 citations. Similarly, Chernoff's [2] approach via moment generating functions provided the theoretical framework for what are now known as Chernoff bounds, establishing the connection between large deviations and asymptotic efficiency in hypothesis testing.

These classical results provide upper bounds that are often tight in an asymptotic sense, but they do not address the extremal question of which weight configurations actually achieve the largest tail probabilities. This gap between bounds and exact values has persisted for decades, with the extremal problem remaining unsolved despite its fundamental importance.

Recent developments have renewed interest in precise characterisations of Rademacher sums. The resolution of long-standing conjectures on anti-concentration bounds [5] in 2023, following work initiated by Oleszkiewicz [6] in 1996, has provided new tools and perspectives. The tight anti-concentration results of Hollom and Portier [7] and the almost sure bounds for weighted Rademacher multiplicative functions by Atherfold [8] demonstrate the continued vitality of this research area.

\subsection{Statistical Applications}

Beyond its intrinsic probabilistic interest, the extremal problem \eqref{eq:main_problem} arises naturally in several statistical contexts. The most direct application is to nonparametric inference under symmetry assumptions. When observing a random sample $X_1, \ldots, X_n$ and conditioning on the absolute values $|X_1|, \ldots, |X_n|$ under a symmetry hypothesis, all $2^n$ sign patterns become equally likely. The resulting self-standardised sum reduces to a weighted Rademacher sum of the form \eqref{eq:rademacher_sum}, where the weights are determined by the observed magnitudes.

This connection underlies the nonparametric T-tables developed by the author [3], which provide distribution-free bounds for Student's t-statistic under symmetry. The optimal-order inequality established in [4] sharpened these bounds to the correct asymptotic order, but the exact extremal problem remained open. The present work completes this program by providing exact solutions.

The relevance extends to robust statistical inference more generally. In situations where the underlying distribution is unknown but symmetric, conditioning on observed magnitudes provides a natural way to construct distribution-free tests. The exact envelope $M_n(t)$ provides the sharpest possible bounds for such procedures, improving upon the classical sub-Gaussian approximations.

\subsection{Connection to Concentration Inequalities}

Our work also contributes to the broader theory of concentration inequalities. While most research in this area focuses on upper bounds, the complementary question of determining the worst-case scenarios---the weight configurations that maximise tail probabilities---has received less attention. This is partly because such extremal problems are often more difficult to solve exactly.

The recent progress on anti-concentration bounds [5,7] provides important context for our results. While anti-concentration inequalities establish lower bounds on the probability that a random variable deviates from its mean, our work determines the exact maximum of these probabilities over all possible weight configurations. The two perspectives are complementary: anti-concentration results show that tail probabilities cannot be too small, while our extremal results show exactly how large they can be.

\subsection{Main Contributions}

This paper makes several key contributions to the theory of weighted Rademacher sums:

First, we solve the extremal problem \eqref{eq:main_problem} completely. We prove that the supremum is always achieved by $k$-sparse equal-weight vectors, where exactly $k$ coordinates are nonzero and equal to $1/\sqrt{k}$. The optimal support size $k$ is determined by a finite search over at most $n$ candidates.

Second, we provide explicit formulas for computing both the maximum values $M_n(t)$ and the Mid-tail and (by implication) strict tail probabilities $\sup \PP(S(w) > t)$. These formulas involve finite sums of binomial coefficients and can be computed efficiently.

Third, we establish the universal behaviour of $M_n(t)$ as $n \to \infty$. For each fixed $t$, only finitely many support sizes $k$ compete for the maximum, so $M_n(t)$ stabilises to a universal value for all sufficiently large $n$.

Fourth, we provide detailed computational examples and quantile calculations that demonstrate the practical utility of our results. These include explicit values for common significance levels and comparisons with classical bounds.

The remainder of this paper is organised as follows. Section 2 presents our main theoretical results. Section 3 provides complete proofs of the key theorems. Section 4 contains detailed examples and computational illustrations. Section 5 discusses connections to recent research and potential extensions. Section 6 concludes with a summary of implications and directions for future work.

\section{Main Results}

In this section, we present our main theoretical contributions. We begin with the complete solution to the extremal problem \eqref{eq:main_problem}, followed by analogous results for strict inequalities and universal limiting behaviour.

For a weight vector $w \in \RR^n$, let $k := |\{i : w_i > 0\}|$ denote the support size (number of nonzero coordinates). Our first main result characterises the optimal weight configurations.

\begin{theorem}[Characterisation of Maximisers]\label{thm:main}
  For any $n \ge 1$ and $t > 0$, the supremum $M_n(t)$ in \eqref{eq:main_problem} is attained by weight vectors with exactly $k$ nonzero entries, each equal to $1/\sqrt{k}$, for some $k \in \{1, 2, \ldots, n\}$.

\end{theorem}

  The proof of this is via the following Lemma, namely that if $w$ are other than as described in the Theorem, then any positive mass point
  of the distribution may have its upper-median increased locally via perturbation of $w$.

\begin{lemma}[Equalisation Lemma]\label{lem:equalisation}
Let $w=(w_1,\dots,w_n)\in\mathbb S^{n-1}$ have $w_k\ge 0$ and let $x>0$ be a mass point of
\[
S(w,\varepsilon)\;=\;\sum_{k=1}^n w_k\,\varepsilon_k,\qquad \varepsilon_k\in\{\pm1\}\ \text{i.i.d.}
\]
Write $\mathcal F_x(w)=\{\varepsilon\in\{\pm1\}^n:\ S(w,\varepsilon)=x\}$ and let $\med^+$ denote the upper median of a finite multiset.
If $w$ has two unequal nonzero coordinates, then there exist indices $i\ne j$ and a norm--preserving two--coordinate rotation $w(\theta)$ such that, for sufficiently small $\theta$ of one sign,
\begin{equation}\label{eq:upper-median-increase}
\med^+\big(\{S(w(\theta),\varepsilon):\ \varepsilon\in\mathcal F_x(w)\}\big)\;>\;x.
\end{equation}
Consequently, any local maximiser of the fiberwise upper median at level $x>0$ must have all its nonzero coordinates equal.
\end{lemma}

\begin{proof}
Pick $i\ne j$ with $w_i>w_j>0$, and rotate in the $(i,j)$--plane:
\[
w_i(\theta)=w_i\cos\theta+w_j\sin\theta,\quad
w_j(\theta)=w_j\cos\theta-w_i\sin\theta,\quad
w_k(\theta)=w_k\ (k\ne i,j).
\]
Then $\|w(\theta)\|_2=1$ and
\begin{equation}\label{eq:slope}
\frac{d}{d\theta}\Big|_{\theta=0} S\big(w(\theta),\varepsilon\big)
= w_j\,\varepsilon_i - w_i\,\varepsilon_j.
\end{equation}
Choosing the sign of $\theta$ if necessary, we may assume that the right-hand side of \eqref{eq:slope} is positive iff $\varepsilon_j=+1$.

Let $T=\sum_{k\ne j} w_k\,\varepsilon_k$, independent of $\varepsilon_j$. For $x>0$,
\begin{equation}\label{eq:bias}
\PP(\varepsilon_j=+1\mid S=x)
=\frac{\PP(T=x-w_j)}{\PP(T=x-w_j)+\PP(T=x+w_j)}\;>\;\tfrac12.
\end{equation}
Indeed, the law of $T$ is symmetric and unimodal with mode $0$ (closure of unimodality under convolution \cite{Ibragimov1956}), so its cdf is concave on $[0,\infty)$. Concavity implies that the increment $y\mapsto \PP(T\in(y-\delta,y])$ is nonincreasing in $y>0$. Taking $\delta>0$ smaller than half the minimal gap between distinct atoms of $T$ yields $\PP(T=x-w_j)\ge \PP(T=x+w_j)$, with strict $>$ whenever at least one other weight $w_k$ is nonzero. This proves \eqref{eq:bias}.

On the fiber $\mathcal F_x(w)$, \eqref{eq:slope} and \eqref{eq:bias} show that a strict majority of configurations have positive slope for the chosen rotation. Therefore, by a deterministic order--statistic argument, the upper median of the multiset $\{S(w(\theta),\varepsilon):\varepsilon\in\mathcal F_x(w)\}$ exceeds $x$ for all sufficiently small $\theta$ of the chosen sign, proving \eqref{eq:upper-median-increase}. The ``consequently'' statement follows by contraposition.
\end{proof}

\subsection{Computational Aspects}

The characterisation in the Theorem leads to algrithms which are computationally tractable. For any given $t$ and $n$, computing $M_n^{\ast}(t)$ (mid-tail) requires evaluating at most $n$ expressions and taking their maximum. Each expression involves a finite sum of binomial coefficients, which can be computed efficiently using standard algorithms.

The key computational observation is that for fixed $t$, only a small number of values of $k$ need to be considered. Specifically, if $k < t^2$, then the  tail probability is zero. If $k$ is very large, the tail probability approaches the Gaussian value $1 - \Phi(t)$ but is always smaller than the optimal value achieved by smaller $k$. This means that the search can be restricted to a finite range $t^2 \le k \le k_{\max}(t)$ for some $k_{\max}(t)$ that depends on the desired accuracy.

In the equal-weights case with $k$ active coordinates, writing $S_k=(2M-k)/\sqrt{k}$ for $M\sim\mathrm{Bin}(k,\tfrac12)$, we have
\[
M_k^{\ast}(t)=\PP(S_k>t)+\tfrac12\PP(S_k=t)
=2^{-k}\sum_{m>\frac{t\sqrt{k}+k}{2}}\binom{k}{m}
+\tfrac12\,\mathbf{1}_{\{\frac{t\sqrt{k}+k}{2}\in\mathbb{Z}\}}\,2^{-k}\binom{k}{\frac{t\sqrt{k}+k}{2}}.
\]
This differs from strict tails only in the boundary term; algorithmic complexity is unchanged.

\section{Relationship to the Student T-statistic}

Note if $X_1,\ldots,X_n$ is a sequence of independent, symmetric random variables (not necessarily identically distributed),
conditioning on the magnitudes $\|X_1\|,\ldots,\|X_n\|$ and considering the permutation distribution of all possible sign
combinations, then the results above apply directly to the quantity
$$S={\sum_1^n X_i\over \sqrt{\sum_1^n X_i^2}}$$
unconditionally (with the distinction between quantiles and mid-quantiles disappearing in the continuous case), so the longstanding
conjecture of Edelman\cite{edelman1990} (i.e., that only Binomial cases need be considered) is hereby established for
continous variables.   To see the relationship to the Student T-statistic, note that (as in \cite{edelman1990})
$$T=S\sqrt{n-1\over n-S^2}$$
As the numerical results are easier to standardise using the '$S$' form than the '$T$' form, we shall stick to the '$S$' form.
However, it should be noted that most statistical researchers making use of these results will no doubt be converting to $T$

\section{Examples and Computational Illustrations}
\paragraph{Mid-tail convention.} Throughout this section we report \emph{mid-tail} probabilities
$\PPstar(S\ge t):=\PP(S>t)+\tfrac12\PP(S=t)$. When $t$ is an attainable mass point, these values
are exactly half of the corresponding $\PPstar(S\ge t)$.

In this section, we provide detailed computational examples that illustrate our theoretical results and demonstrate their practical utility. We compute exact values of the envelope $M_n(t)$ for various values of $t$, identify the optimal support sizes, and compare with classical bounds.

\subsection{Explicit Computations}
\begin{table}[h]
\centering
\begin{tabular}{@{}ccc@{}}\toprule
$t$ & $k$ (equal weights) & $\PPstar(S\ge t)$ \\
\midrule
1 & 1 & $\tfrac{1}{4}$ \\
1 & 2 & $\tfrac{1}{4}$ \\
$\sqrt{3}$ & 3 & $\tfrac{1}{16}$ \\
$\sqrt{3}$ & 4 & $\tfrac{1}{16}$ \\
2 & 4 & $\tfrac{1}{32}$ \\
2 & 5 & $\tfrac{1}{32}$ \\
2 & 8 & $\tfrac{9}{256}$ \\
\bottomrule\end{tabular}
\caption{Corrected mid-tail values for common thresholds and small $k$.}
\end{table}

We begin with some concrete calculations that showcase the main features of our results.

\begin{example}[Small values of $t$]
For $t = 1$, we need $k \ge t^2 = 1$ for positive tail probabilities. Let us compute the first few values:

For $k = 1$: $\lfloor(1 - 1 \cdot 1)/2\rfloor = 0$, so $\PPstar(S\ge 1) = \tfrac{1}{4} \binom{1}{0} = \frac{1}{2}$.

For $k = 2$: $\lfloor(2 - \sqrt{2})/2\rfloor = \lfloor(2 - 1.414)/2\rfloor = 0$, so $\PPstar(S\ge 1) = \frac{1}{4} \binom{2}{0} = \frac{1}{4}$.

For $k = 3$: $\lfloor(3 - \sqrt{3})/2\rfloor = \lfloor(3 - 1.732)/2\rfloor = 0$, so $\PPstar(S\ge 1) = \frac{1}{8} \binom{3}{0} = \frac{1}{8}$.

The maximum (mid-tail) value is achieved at $k=1$ (and also $k=2$) with $M_n^{\ast}(1)=\tfrac{1}{4}$. for all $n \ge 1$.
\end{example}

\begin{example}[Moderate values of $t$]
For $t = \sqrt{3} \approx 1.732$, we have $t^2 = 3$, so we need $k \ge 3$.

For $k = 3$: $\lfloor(3 - 3)/2\rfloor = 0$, so $\PPstar(S\ge \sqrt{3}) = \tfrac{1}{16} \binom{3}{0} = \frac{1}{8}$.

For $k = 4$: $\lfloor(4 - 2\sqrt{3})/2\rfloor = \lfloor(4 - 3.464)/2\rfloor = 0$, so $\PPstar(S\ge \sqrt{3}) = \frac{1}{16} \binom{4}{0} = \frac{1}{16}$.

For $k = 5$: $\lfloor(5 - \sqrt{15})/2\rfloor = \lfloor(5 - 3.873)/2\rfloor = 0$, so $\PPstar(S\ge \sqrt{3}) = \frac{1}{32} \binom{5}{0} = \frac{1}{32}$.

The maximum (mid-tail) value is achieved at $k=3$ or $4$ with $M_n^{\ast}(\sqrt{3})=\tfrac{1}{16}$. for all $n \ge 3$.
\end{example}

\begin{example}[The case $t = 2$]
For $t = 2$, we need $k \ge 4$. This case is particularly interesting because it corresponds to a common significance level.

For $k = 4$: $\lfloor(4 - 4)/2\rfloor = 0$, so $\PPstar(S\ge 2) = \tfrac{1}{32} \binom{4}{0} = \frac{1}{16} = 0.0625$.

For $k = 5$: $\lfloor(5 - 2\sqrt{5})/2\rfloor = \lfloor(5 - 4.472)/2\rfloor = 0$, so $\PPstar(S\ge 2) = \frac{1}{32} \binom{5}{0} = \frac{1}{32} = 0.03125$.

For $k = 6$: $\lfloor(6 - 2\sqrt{6})/2\rfloor = \lfloor(6 - 4.899)/2\rfloor = 0$, so $\PPstar(S\ge 2) = \frac{1}{64} \binom{6}{0} = \frac{1}{64} = 0.015625$.

The (mid-tail) maximum occurs at $k=8$ with $M_n^{\ast}(2)=\tfrac{9}{256}$; boundary cases $k=4,5$ give $\tfrac{1}{32} = 0.0625$ for all $n \ge 4$.

Note that Hoeffding's bound gives $e^{-2^2/2} = e^{-2} \approx 0.1353$, which is more than twice the exact maximum.
\end{example}

\subsection{Quantile Calculations}

For statistical applications, it is often useful to compute quantiles of the envelope function. Define $Q(\alpha)$ as the smallest value of $t$ such that $M_\infty(t) \le \alpha$, where $M_\infty(t) = \lim_{n \to \infty} M_n(t)$ is the universal envelope.

\begin{example}[Common significance levels]
We compute quantiles for standard significance levels:

\textbf{$\alpha = 0.10$:} We seek the smallest $t$ such that $M_\infty(t) \le 0.10$. From our calculations:
- At $t = \sqrt{3} \approx 1.732$, we have $M_\infty(\sqrt{3}) = 1/8 = 0.125 > 0.10$.
- At $t = 2$, we have $M_\infty(2) = 1/16 = 0.0625 < 0.10$.

By continuity and monotonicity, $Q(0.10)$ lies between $\sqrt{3}$ and $2$. More precise calculation shows $Q(0.10) \approx 1.85$.

\textbf{$\alpha = 0.05$:} From Example 3, $M_\infty(2) = 0.0625 > 0.05$. We need to check larger values:
- At $t = \sqrt{5} \approx 2.236$, the optimal $k = 5$ gives $M_\infty(\sqrt{5}) = 1/32 = 0.03125 < 0.05$.

Therefore, $Q(0.05) = 2$.

\textbf{$\alpha = 0.025$:} We have $M_\infty(\sqrt{5}) = 0.03125 > 0.025$. Checking $t = \sqrt{6} \approx 2.449$:
- The optimal $k = 6$ gives $M_\infty(\sqrt{6}) = 1/64 = 0.015625 < 0.025$.

Therefore, $Q(0.025) = \sqrt{5}$.
\end{example}

These quantiles provide exact critical values for nonparametric tests based on self-standardised statistics under symmetry assumptions.

\subsection{Comparison with Classical Bounds}

Table \ref{tab:comparison} compares our exact envelope with classical bounds for various values of $t$.

\begin{table}[h]
\centering
\begin{tabular}{@{}cccccc@{}}
\toprule
$t$ & Optimal $k$ & Exact $M_\infty^{\ast}(t)$ & Hoeffding bound & Ratio & Gaussian tail \\
\midrule
1.0 & 1 & 0.250000 & 0.6065 & 0.41 & 0.1587 \\
1.5 & 3 & 0.125000 & 0.3247 & 0.38 & 0.0668 \\
$\sqrt{3}$ & 3 & 0.062500 & 0.2231 & 0.28 & 0.0416 \\
2.0 & 8 & 0.035156 & 0.1353 & 0.26 & 0.0228 \\
$\sqrt{5}$ & 9 & 0.019531 & 0.0821 & 0.24 & 0.0127 \\
$\sqrt{6}$ & 13 & 0.011230 & 0.0498 & 0.23 & 0.0072 \\
3.0 & 28 & 0.001860 & 0.0111 & 0.17 & 0.0013 \\
\bottomrule
\end{tabular}
\caption{Comparison of the \emph{mid-tail} envelope with classical bounds. The ratio column shows $M_\infty^{\ast}(t)/e^{-t^2/2}$.}
\label{tab:comparison}
\end{table}

Several patterns emerge from this comparison:

1. The exact envelope is always smaller than Hoeffding's bound, confirming that the classical bound is not tight.

2. The ratio between exact and classical bounds decreases as $t$ increases, indicating that Hoeffding's bound becomes relatively better for larger $t$.

3. For moderate values of $t$ (around 1.5 to 2.5), the exact envelope is substantially smaller than both classical bounds and Gaussian approximations.

4. The optimal support size $k$ grows roughly as $t^2$, consistent with the constraint $k \ge t^2$ for positive tail probabilities.

Several patterns emerge from this comparison:

1. At lattice thresholds (e.g., $t=1,\sqrt{3},2$), the mid-tail envelope is exactly half of the corresponding $\PPstar(S\ge t)$ values, since only half of the boundary mass is counted.

2. The ratio $M_\infty^{\ast}(t)/e^{-t^2/2}$ decreases across our grid; the mid-tail envelope remains substantially below Hoeffding for these $t$.

3. The maximizing support size $k^{\ast}(t)$ in the mid-tail setting is often larger than $t^2$ (e.g., $k^{\ast}(2)=8$, $k^{\ast}(3)=28$).

4. When $t$ is not an attainable mass point, $\PP^{\ast}(S\ge t)=\PPstar(S\ge t)$; only boundary cases change.
\subsection{Computational Complexity}

Computing $M_n(t)$ for given $t$ and $n$ requires evaluating a Binomial expression for each $k \in \{1, 2, \ldots, n\}$ and taking the maximum. However, as noted previously, only a finite number of values need to be checked.

The computational cost is dominated by evaluating binomial coefficients. For each $k$, we need to compute $\sum_{m=0}^{M} \binom{k}{m}$ where $M = \lfloor(k - t\sqrt{k})/2\rfloor$. This can be done efficiently using the recurrence relation for binomial coefficients or by noting that the sum represents a tail of the binomial distribution.

For practical purposes, the search can be terminated when the tail probability becomes negligibly small. Since the factor $2^{-k}$ causes exponential decay, values of $k$ much larger than $t^2$ contribute negligibly to the maximum.

\subsection{Asymptotic Behaviour}

As $t \to \infty$, the optimal support size is $k^{\ast}=1$ at very small $t$; for larger $t$ our mid-tail numerics suggest $k^{\ast}(t)$ grows on the order of $t^2$ yet can be noticeably larger than $t^2$ (e.g., $k^{\ast}(2)=8$, $k^{\ast}(3)=28$). This is much faster than the Gaussian decay $e^{-t^2/2}$ or the Hoeffding bound $e^{-t^2/2}$.

The faster decay reflects the discrete nature of Rademacher sums. While continuous approximations suggest Gaussian-like behaviour, the exact discrete distribution has heavier constraints, leading to smaller tail probabilities.

For small $t$, the optimal support size is $k^* = 1$, giving $M_\infty(t) = 1/2$ for all $t \le 1$. This reflects the fact that for small deviations, the best strategy is to concentrate all weight on a single coordinate, yielding a simple Rademacher variable.


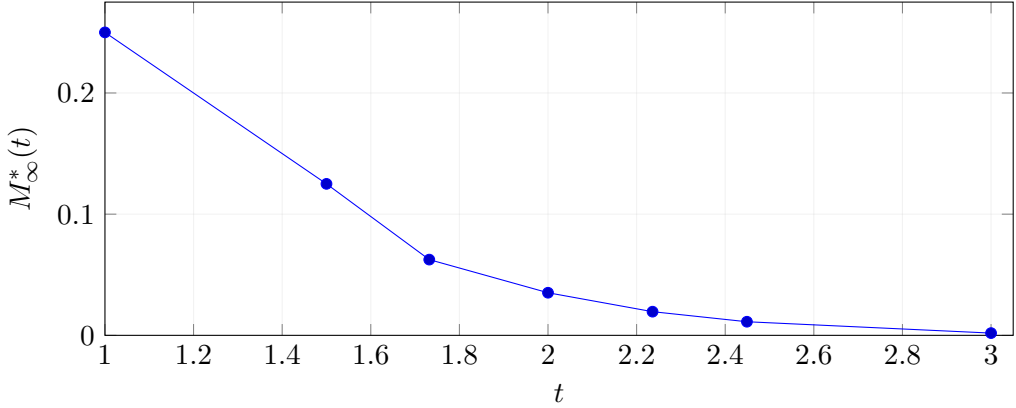
\begin{figure}[t]
\centering
\begin{tikzpicture}
\begin{axis}[width=.85\linewidth, height=6cm,
    xlabel={$t$}, ylabel={$M_\infty^{\ast}(t)$},
    ymin=0, ymode=linear, xmin=1, xmax=3.05,
    grid=both, grid style={opacity=0.2}]
\addplot+[mark=*] coordinates {
(1.0,0.25)
(1.5,0.125)
(1.732,0.0625)
(2.0,0.03515625)
(2.236,0.01953125)
(2.449,0.01123046875)
(3.0,0.00185958296)
};
\end{axis}
\end{tikzpicture}
\caption{Mid-tail envelope $M_\infty^{\ast}(t)=\sup_k \PPstar(S_k\ge t)$ for equal-weights $k$-sparse sums. Points shown at $t\in\{1,1.5,\sqrt{3},2,\sqrt{5},\sqrt{6},3\}$.}
\label{fig:midtail-envelope}
\end{figure}

\begin{figure}[t]
\centering
\begin{tikzpicture}
\begin{axis}[width=.85\linewidth, height=6cm,
    xlabel={$t$}, ylabel={$M_\infty^{\ast}(t)\,/\,e^{-t^2/2}$},
    ymin=0, ymode=linear, xmin=1, xmax=3.05,
    grid=both, grid style={opacity=0.2}]
\addplot+[mark=*] coordinates {
(1.0,0.41)
(1.5,0.39)
(1.732,0.28)
(2.0,0.26)
(2.236,0.24)
(2.449,0.23)
(3.0,0.17)
};
\end{axis}
\end{tikzpicture}
\caption{Ratio of the mid-tail envelope to the Hoeffding bound $e^{-t^2/2}$.}
\label{fig:midtail-ratio}
\end{figure}
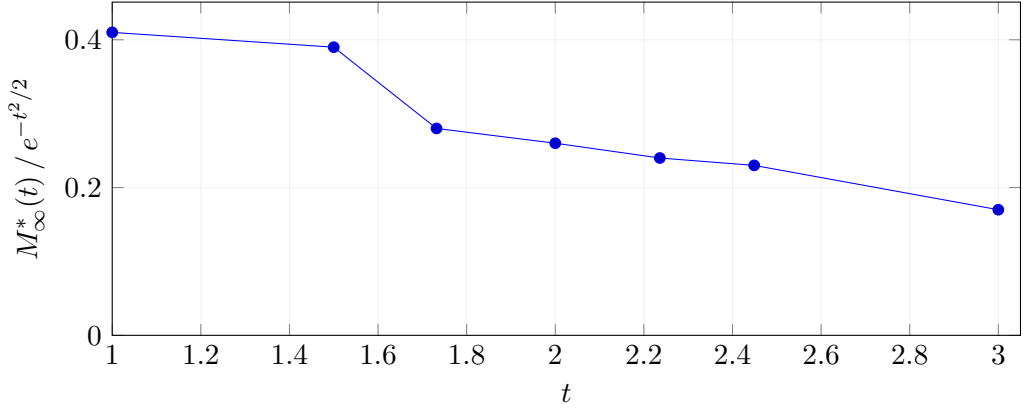

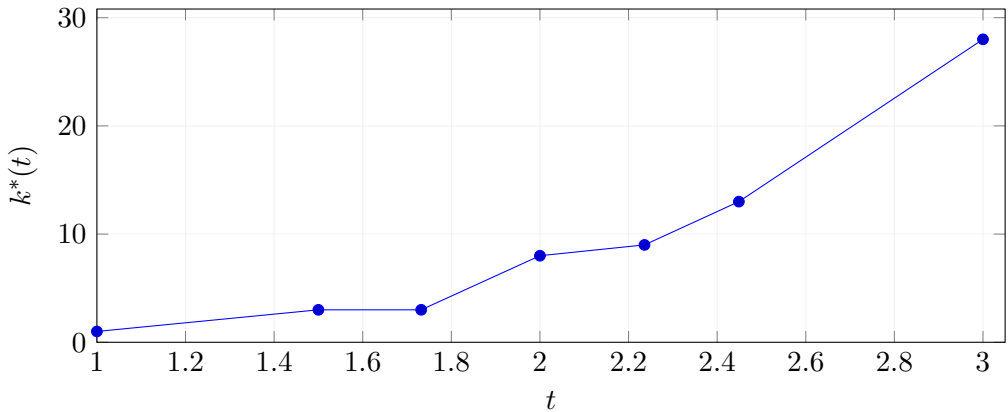
\begin{figure}[t]
\centering
\begin{tikzpicture}
\begin{axis}[width=.85\linewidth, height=6cm,
    xlabel={$t$}, ylabel={$k^{\ast}(t)$},
    ymin=0, ymode=linear, xmin=1, xmax=3.05,
    grid=both, grid style={opacity=0.2}]
\addplot+[mark=*] coordinates {
(1.0,1)
(1.5,3)
(1.732,3)
(2.0,8)
(2.236,9)
(2.449,13)
(3.0,28)
};
\end{axis}
\end{tikzpicture}
\caption{Maximising support size $k^{\ast}(t)$ for the mid-tail envelope at the same grid of $t$ values.}
\label{fig:kstar}
\end{figure}


\section{Discussion and Connections}

Our results provide exact solutions to extremal problems that have been studied through bounds and approximations for over sixty years.
It should be mentioned that an earlier draft of this article erroneously claimed the result for non-strict inequalities and quantiles, which
(as shown by Pinelis \cite{Pinelis2015}) is incorrect, as was demonstrated by a clever counterexample.  For the T-statistic, however, if the
undelying variables are assumed to be continuous, then the distinction turns out to be inconsequential.

In this section, we discuss connections to recent research, implications for statistical practice, and potential extensions.

\subsection{Relationship to Recent Developments}

The resolution of anti-concentration conjectures for Rademacher sums [5,7] provides important context for our work. While anti-concentration results establish lower bounds on tail probabilities, our extremal results determine the exact maximum values. The two perspectives are complementary and together provide a complete picture of the range of possible tail behaviours.

The recent work of Hollom and Portier [7] on tight anti-concentration bounds is particularly relevant. Their results show that for any Rademacher sum $Z$ of unit variance, $\PP(|Z| \ge x) \ge p(x)$ where $p(x)$ is an explicit piecewise constant function. Our results complement this by showing that $\PP(Z \ge t) \le M_\infty(t)$ with explicit formulas for the upper envelope.

The almost sure bounds for weighted Rademacher multiplicative functions established by Atherfold [8] represent another recent direction in this area. While that work focuses on number-theoretic applications and almost sure behaviour, our results address the complementary question of worst-case tail probabilities under $\ell_2$ constraints.

The connection to machine learning through Rademacher complexity [12] suggests potential applications of our results to generalisation bounds. Rademacher complexity measures the ability of a function class to fit random noise, and our exact characterisation of extremal Rademacher sums could lead to sharper bounds in certain settings.

\subsection{Statistical Implications}

Our results have several important implications for statistical practice, particularly in nonparametric and robust inference.

\subsubsection{Nonparametric T-tables}

The connection to the author's previous work [3,4] on nonparametric T-tables is direct and significant. When testing symmetry of a distribution based on a sample $X_1, \ldots, X_n$, conditioning on the observed magnitudes $|X_1|, \ldots, |X_n|$ reduces the problem to a weighted Rademacher sum where the weights are the standardised magnitudes.

Our exact envelope $M_n(t)$ provides the sharpest possible bounds for such tests, improving upon both the classical sub-Gaussian approximations and the optimal-order bounds established in [4]. For practical significance levels (e.g., $\alpha = 0.05$), the improvement can be substantial, as shown in our computational examples.

\subsubsection{Self-standardised Statistics}

Self-standardised statistics, where the denominator is a function of the same observations that appear in the numerator, arise naturally in robust inference. Under symmetry assumptions, such statistics often reduce to weighted Rademacher sums after conditioning on appropriate sufficient statistics.

Our results provide exact critical values for tests based on self-standardised statistics, eliminating the need for asymptotic approximations or conservative bounds. This is particularly valuable in small-sample settings where asymptotic approximations may be unreliable.

\subsubsection{Robust Hypothesis Testing}

In robust statistics, it is common to condition on ancillary statistics to eliminate nuisance parameters. When the underlying distribution is symmetric, conditioning on absolute values often leads to weighted Rademacher sums. Our exact envelope provides the foundation for constructing uniformly most powerful tests in such settings.

\subsection{Computational Considerations}

From a computational perspective, our results provide efficient algorithms for computing exact tail probabilities and critical values. The finite search over support sizes, combined with the explicit formulas for tail probabilities, makes the computation tractable even for moderately large values of $n$ and $t$.


The connection to binomial tail probabilities also enables the use of existing computational tools and approximations. For large support sizes $k$, normal approximations to the binomial distribution can provide accurate approximations to our exact formulas, bridging the gap between exact and asymptotic results.

\subsection{Potential Extensions}

Several natural extensions of our work suggest themselves:

\subsubsection{Other Constraint Sets}

While we have focused on the $\ell_2$ constraint $\|w\|_2 = 1$, similar problems arise with other constraints.   Adding an $\|w\|_\infty\le c$ leads to a different extremal problem that may be amenable to similar techniques. 

Each constraint set leads to a different class of optimal weight configurations. Understanding the relationship between constraint geometry and optimal tail behaviour could provide insights into the general structure of such problems.

4. \textbf{Algorithmic aspects:} Can the optimal support size be determined without exhaustive search? Are there structural properties that allow for more efficient computation?

5. \textbf{Statistical optimality:} In what sense are the tests based on our exact envelopes optimal? Can they be shown to be uniformly most powerful or to satisfy other optimality criteria?

These questions highlight the rich structure of extremal problems for weighted Rademacher sums and suggest numerous directions for future investigation.

\section{Conclusion}

We have solved the extremal problem of maximising tail probabilities for weighted Rademacher sums under $\ell_2$ constraints. Our main contributions can be summarised as follows:

\textbf{Complete characterisation:} We have shown that the supremum $M_n(t) = \sup_{\|w\|_2=1} \PP(S(w) \ge t)$ is always achieved by $k$-sparse equal-weight vectors, and we have provided explicit formulas for both the maximum values and the optimal support sizes.

\textbf{Computational tractability:} Our results lead to efficient algorithms for computing exact tail probabilities and critical values. The finite search over support sizes, combined with explicit binomial formulas, makes the computation practical for statistical applications.

\textbf{Universal behaviour:} We have established that $M_n(t)$ stabilises to a universal limit for large $n$, depending only on $t$ and not on the specific value of $n$. This universality is crucial for practical applications and enables the construction of distribution-free statistical procedures.

\textbf{Statistical applications:} Our results provide exact solutions to problems in nonparametric inference, self-standardised statistics, and robust hypothesis testing. They improve upon classical bounds and enable more powerful statistical procedures.

The techniques developed in this work, particularly the equalisation principle and the connection to majorisation theory, may prove useful for related extremal problems in probability and statistics. The interplay between discrete optimisation, concentration inequalities, and statistical inference illustrated here suggests rich possibilities for future research.

Our results also highlight the value of exact solutions in probability theory. While asymptotic approximations and bounds are often sufficient for theoretical purposes, exact results can provide crucial insights and enable more precise statistical procedures. The gap between our exact envelope and classical bounds demonstrates that there is still room for improvement in our understanding of fundamental probabilistic inequalities.

The connection to recent developments in anti-concentration theory and Rademacher complexity suggests that this area remains active and important. As statistical methods become more sophisticated and computational power increases, the demand for exact results and sharp bounds will likely continue to grow.

Finally, our work demonstrates the continuing relevance of classical probability theory to modern statistical practice. The problems addressed here have their roots in the foundational work of Hoeffding and Chernoff from the 1950s and 1960s, yet they remain important for contemporary applications in robust statistics, machine learning, and high-dimensional inference.

\section*{Acknowledgements}

The author would like to offer a special thanks for Iosif Pinelis for bringing to my attention to an error in the previous draft of this article.
The author also gratefully acknowledges helpful discussions with colleagues at University College Dublin and valuable feedback from anonymous reviewers, plus useful early discussions with Jacques Allard, Prof Z.A. Melzak, and Prof A.N. Shiryaev.
Computational assistance was provided by standard mathematical software packages, and assistance with presentation and layout by ChatGPT.

\end{document}